\documentclass[10pt]{article}
\usepackage{}
\usepackage{amsmath}
\usepackage{amsfonts}
\usepackage{graphicx}
\usepackage{amssymb}

\newtheorem{condition**}{A*}
\newtheorem{condition***}{C*}
\newtheorem{condition*}{C}

\newtheorem{prop}{Propsition}[section]

\newtheorem{definition}{Definition}[section]
\newtheorem{theorem}{Theorem}[section]
\newtheorem{lemma}{Lemma}[section]
\newtheorem{remark}{Remark}[section]

\newenvironment{keywords}{{\bf Key words: }}{}

\begin{document}

\title{A Class of Mean-field LQG Games with Partial Information}

\author{Jianhui Huang\thanks{Corresponding Author. J. Huang is with the Department of Applied Mathematics, The Hong Kong Polytechnic University, Hong Kong
(majhuang@polyu.edu.hk).}
\quad \quad Shujun Wang \thanks{S. Wang is with the Department of Applied Mathematics, The Hong Kong Polytechnic University, Hong Kong
(shujun.wang@connect.polyu.hk). The authors acknowledge the financial support from RGC Earmarked grant 500909.}}

\maketitle

\begin{abstract}
The large-population system consists of considerable small agents whose individual behavior and mass effect are interrelated via their state-average. The mean-field game provides an efficient way to get the decentralized strategies of large-population system when studying its dynamic optimizations. Unlike other large-population literature, this current paper possesses the following distinctive features. First, our setting includes the partial information structure of large-population system which is practical from real application standpoint. Specially, two cases of partial information structure are considered here: the partial filtration case (see Section 2, 3) where the available information to agents is the filtration generated by an observable component of underlying Brownian motion; the noisy observation case (Section 4) where the individual agent can access an additive white-noise observation on its own state. Also, it is new in filtering modeling that our sensor function may depend on the state-average. Second, in both cases, the limiting state-averages become random and the filtering equations to individual state should be formalized to get the decentralized strategies. Moreover, it is also new that the limit average of state filters should be analyzed here. This makes our analysis very different to the full information arguments of large-population system. Third, the consistency conditions are equivalent to the wellposedness of some Riccati equations, and do not involve the fixed-point analysis as in other mean-field games. The $\epsilon$-Nash equilibrium properties are also presented.
\end{abstract}

\\

\section{Introduction}

The starting point of our work is the large-population systems which are strongly grounded in various fields. The most significant feature of large-population systems is the existence of a large number of negligible agents that are coupled in their dynamics and (or) cost functionals via state average. Due to this highly complicated coupling feature, it is intractable for the agents to study the centralized strategies of large-population system. Instead, it is more appropriate to discuss the decentralized strategies which encompass the individual state and some off-line quantities only. For this purpose, one efficient methodology is the associated mean-field games which enable us to obtain the decentralized control, in its general nonlinear setting, through some Hamilton-Jacobi-Bellman (HJB) equation coupled with Fokker-Planck (FP) equation. The interested readers may refer the pioneering work \cite{ll} for the motivation and methodology of mean-field games. Based on \cite{ll}, considerable research attention has been drawn along this research line. Some recent literature include \cite{B12,B11,dh,hcm07,HCM12,hmc06,LZ08} for linear-quadratic-Gaussian (LQG) mean-field games of large-population system, \cite{H10,NH12} for large population systems with major and minor players, \cite{TZB11} for risk sensitive mean-field games. In addition, the stochastic control problems with a mean-field term in dynamics and (or) cost functional can be found in \cite{AD,BDL,MOZ,yong 2011} etc.\

This paper focuses on the study of large-population system in its linear-quadratic (LQ) case (linear state, quadratic functional) by taking into account the partial information structure. Recall the linear stochastic system and its related LQ control have already been extensively investigated. One systematic introduction of stochastic LQ control can be found in the monograph \cite{yz} and the references therein. It turns out that various stochastic LQ control problems fit into the partial information framework due to the factors such as finite datum, latent process or noisy observation, etc. An extensive review of LQ control with partial information is provided in \cite{Ben92} and other related works include \cite{BO07,BEK,Ben83,Tang98,Zhou93} etc. Herein, we turn to study the partial information structures of linear large-population systems. Two cases of partial information structure are addressed and discussed in our present paper. In the first case, the individual agents can only access the filtration generated by one observable component of underlying Brownian motion. The unobservable Brownian motion component may be interpreted as the effect of a passive version of major player in the context of \cite{H10,NH12}, or be framed into a partial observation problem (see \cite{Ben92}) in which the sensor function is of zero dynamics. We discuss the related mean-field LQG games and derive the decentralized strategies which are shown to satisfy the $\epsilon$-Nash equilibrium property. An approximation scheme is introduced here and the limiting state-average process is actually driven by the unobservable component of Brownian motion. The consistency condition is further obtained through some Riccati equations. Some auxiliary mean-field stochastic differential equation (MFSDE) is introduced and analyzed. In the second case, the individual agents in our large-population system can access the information on its state by an additive white-noise observation. The decentralized strategies are derived with the help of Kalman filtering equation to the underlying state. Note that due to intrinsic approximation scheme of large-population system, the ``innovation" process and related Kalman filtering equation derived here are not defined in their classical sense. The consistency condition is obtained via two (coupled) Riccati equations. It is notable a class of mean-field LQG games with noisy observations is also addressed in \cite{hcm06} but defined on an infinite-time horizon so only algebra Riccati equations are involved. Moreover, the limiting state-average in \cite{hcm06} is deterministic as there has no common noise.\

The rest of this paper is organized as follows. Section 2, 3 study the mean-field LQG games with partial filtration structure. Specially, Section 2 gives the problem formulation, discusses the related filtering equation and consistency conditions. Section 3 is devoted to the related $\epsilon$-Nash equilibrium. Section 4 turns to discuss the mean-field LQG games with noisy observation structure. The decentralized strategies are derived by a Kalman filtering equation but in its large-population sense, and the consistency condition is represented via two Riccati equations.

\section{Mean-Field LQG (MFLQG) Games with Partial Filtration}
 Consider a finite time horizon $[0,T]$ for fixed $T>0$. For given filtration $\{\mathcal{F}_t\}_{0 \leq t \leq T},$ let $L^{2}_{\mathcal{F}}(0, T; \mathbb{R})$ denote the space of all $\mathcal{F}_t$-progressively measurable real-valued processes satisfying $\mathbb{E}\int_0^{T}|x(t)|^{2}dt<\infty;$ $L^{2}(0, T; \mathbb{R})$ the space of all real functions satisfying $\int_0^{T}|x(t)|^{2}dt<\infty;$ $L^{\infty}(0, T; \mathbb{R})$  the space of uniformly bounded functions. For a given vector or matrix $M$, $M'$ stands for its transpose. $(\Omega,
\mathcal F, P)$ is a complete probability space on which a standard $(N+1)$-dimensional Brownian motion $\{W(t),W_i(t),\ 1\le i\leq N\}_{0 \leq t \leq T}$ is defined. The information structure of large-population system is as follows. We denote by $\{\mathcal F^{w_i}_t\}_{0\leq t\leq T}$ the filtration generated by the component $W_i;$ $\{\mathcal F^w_t\}_{0 \leq t \leq T}$ the filtration generated by the component $W$. Here, $\{\mathcal F^{w_i}_t\}_{0\leq t\leq T}$ stands for the individual information owning by the $i^{th}$ agent; $\{\mathcal F^w_t\}_{0\leq t\leq T}$ the common information taking effects on all agents. $\mathcal F^i_t:=\sigma(\mathcal F^{w_i}_t \cup \mathcal F^w_t)$ represents the full information of $i^{th}$ agent up to $t$, $\mathcal F_t:=\sigma(\cup_{i=1}^N\mathcal F^i_t)$ denotes the complete information of large-population system up to $t$. For simplicity, we set $\mathcal F=\mathcal F_T$. In decentralized setup, it is infeasible for the $i^{th}$ agent to access the information of other agents, i.e., $\{\mathcal{F}^{w_j}_t\}_{0 \leq t \leq T}$ for $j \neq i.$ This is reasonable due to the asymmetric information (for example, the individual firm's own operation information will not be released to the public or its peer firms). $\{\mathcal F^w_t\}_{0\leq t\leq T}$ can represent the information of some macro process imposing on all agents (firms) due to the common external economic factors. Such information structure can be attributed to a passive major player (e.g., the dominating raw-material supplier that affects the business of all production firms with the same type). Another motivation is from a competitive market consisting of a large number of small participants with principle-agent contract setup. For instance, the contracted farms for planting and harvesting. It also denotes the hidden actions of one monopoly or public environment.\

We consider a large-population system with $N$ individual agents $\{\mathcal{A}_{i}\}_{1 \leq i \leq N}$. The state $x_i$ for $\mathcal{A}_{i}$ satisfies the following controlled linear stochastic system:
\begin{equation}\label{X1}
dx_i(t)=[A(t)x_i(t)+B(t)u_i(t)+\alpha x^{(N)}(t)+m(t)]dt +\sigma(t)dW_i(t)+\tilde{\sigma}(t)dW(t), \quad x_i(0)=x
\end{equation}
where $x^{(N)}(t)=\frac{1}{N}\sum_{i=1}^{N}x_{i}(t)$ is the state-average, $\alpha\in \mathbb{R}$ is the coupling constant. Here, for simplicity, we assume all agents are uniform (homogeneous) with the same coefficients $(A, B, \alpha, m, \sigma, \tilde{\sigma})$ and deterministic initial states $x$. Other works discussing the large-population system with common noise $W$ include \cite{cfs}, \cite{dc}, \cite{gll}. As the first partial information structure, here we assume that $W$ is a latent process which is unobservable to individual agents. Thus, the admissible control $u_i\in \mathcal{U}_i$ where the admissible control set $\mathcal{U}_i$ is defined by$$
\mathcal{U}_i:=\{u_i(\cdot)|u_i(\cdot)\in L^{2}_{\mathcal{F}_t^{w_i}}(0, T; \mathbb{R})\},\ 1\leq i \leq N.
$$\begin{remark}\emph{One explanation of above partial filtration structure comes from portfolio selection. Suppose the financial market consists of a risky-free asset $dS_0=rS_0dt$ and $m$ risky assets (stocks) satisfying geometric Brownian motion: $dS^{i}_t=S^{i}_t(\mu^{i} dt+b^{i} dW_{i}(t))$ with deterministic return rate $\mu^{i}$ and volatility rate $b^{i}$ for $i=1, \cdots, m.$ In its general setup, the wealth process for individual agent follows the \emph{budget equation} (e.g., see \cite{KPQ}): $dX_i=(rX_i-c(t)+\sum_{j=1}^{m}\pi_{i,j}(t)(\mu^j-r))dt+\sum_{j=1}^{m}(\sigma+b^j\pi_{i,j})dW_j+\widetilde{\sigma}dW(t).$ This is an extended equation of \eqref{X1} with control-dependent diffusion term. Here, $W$ denotes the capital gain tax change or taxation exemption (see \cite{cfs,TST}) which is common to all investors but unobservable during the current investment horizon $[0, T]$. The state-average $X^{(N)}$ enters the dynamics or cost functional of $X_i$ when considering the \emph{relative performance criteria} (see \cite{et}). Following \cite{KJ,xz}, in this case, the agent should choose its optimal consumption and investment strategies $u_i=(c_i, \pi_i)$ based on the observable stock prices $\mathcal{F}^{S}_t=\sigma(\cup_{j=1}^{m}\mathcal{F}^{w^{j}}_t)$ only. Further, when considering the \emph{hyperplane investment} or separation market (see \cite{et}): $m>>N,$ or $m=N$, the admissible control of individual agent can be set to be $\mathcal{F}^{w^{i}}_t.$}\end{remark}

\begin{remark}\emph{Another explanation of above partial filtration structure is as follows. Denote by $x_i$ an individual state variable (e.g., the underlying asset value) of a multinational corporation (MNC) whose business includes the domestic part and overseas part. Its asset from domestic business is governed by the idiosyncratic randomness $W_i$ whereas its overseas part mainly relies on the external international economic factors $W$ (i.e., the international raw-oil price which is common to all motor firms). Suppose $x_i$ is unmarketable thus it is unobservable to the market investors. Instead, the investor can only access its marketable domestic part, denoted by $S^{i}_t$, which is given by the geometric Brownian motion model $dS^{i}_t=S^{i}_t(\mu^{i} dt+\sigma^{i} dW^{i}_t).$ Therefore, $\mathcal{F}^{S^{i}}_t=\mathcal{F}^{w^{i}}_t$ and the investors can make their decisions based on $\mathcal{F}^{w^{i}}_t$ hence the partial filtration structure arises.}\end{remark}

Let $u_{-i}=(u_1, \cdots, u_{i-1},$ $u_{i+1}, \cdots u_{N})$ the strategies set of all agents except $\mathcal{A}_i.$ The cost functional of $\mathcal{A}_i$ is
\begin{align}\label{J1}
\mathcal{J}_i(u_i(\cdot),u_{-i}(\cdot))=\ &\mathbb{E}\left[\int_0^T \left(Q(t)(x_i(t)-x^{(N)}(t))^2
 + R(t) u^2_i(t)\right) dt+Gx^2_i(T)\right].
\end{align}Moreover, we have the following assumption:
\begin{equation*}
(H1)\quad
\begin{aligned}
&A(\cdot),B(\cdot),m(\cdot),\sigma(\cdot),\tilde{\sigma}(\cdot),Q(\cdot),R(\cdot)\in L^\infty(0,T;\mathbb{R}),\\
&\alpha\in \mathbb{R},Q(\cdot)\geq0, R(\cdot)>0, G\geq0 .
\end{aligned}
\end{equation*}Now, we formulate the large population LQG games with partial filtration (\textbf{PF}).\\

\textbf{Problem (PF).}
Find a control strategies set $\bar{u}=(\bar{u}_1,\bar{u}_2,\cdots,\bar{u}_N)$ which satisfies
$$
\mathcal{J}_i(\bar{u}_i(\cdot),\bar{u}_{-i}(\cdot))=\inf_{u_i(\cdot)\in \mathcal{U}_i}\mathcal{J}_i(u_i(\cdot),\bar{u}_{-i}(\cdot))
$$where $\bar{u}_{-i}$ represents $(\bar{u}_1,\cdots,\bar{u}_{i-1},\bar{u}_{i+1},\cdots, \bar{u}_N)$.
To study (\textbf{PF}), one efficient protocol is the mean-field LQG games which bridges the ``centralized" LQG problems via the limiting state-average, as the number of agents tends to infinity. To this end, we need figure out the representation of limiting process using heuristic arguments. Based on it, we can find the decentralized strategies by consistency condition. Due to partial filtration structure, it is natural to set the following feedback control on state filters\begin{equation}\label{u1}
\begin{aligned}
\hat{u}_i(t)=-a(t)\mathbb{E}(x_i(t)|\mathcal{F}^{w_i}_t)+\sum_{j=1,j\neq i}^{N}\tilde{a}(t)\mathbb{E}(x_j(t)|\mathcal{F}^{w_i}_t)+b(t)
\end{aligned}
\end{equation}
where the coefficients $a(\cdot), \tilde{a}(\cdot)$ and $b(\cdot)$ are deterministic functions and $\tilde{a}(\cdot)=O(\frac{1}{N})$. Inserting \eqref{u1} into state equation \eqref{X1}, we get the following realized state dynamics
\begin{equation}\label{X2}\begin{aligned}
dx_i(t)=&[A(t)x_i(t)-B(t)a(t)\mathbb{E}(x_i(t)|\mathcal{F}^{w_i}_t)+B(t)\tilde{a}(t)\sum_{j=1,j\neq i}^{N}\mathbb{E}(x_j(t)|\mathcal{F}^{w_i}_t)+B(t)b(t)\\
&+\alpha x^{(N)}(t)+m(t)]dt+\sigma(t)dW_i(t)+\tilde{\sigma}(t)dW(t), \ 1\leq i \leq N.
\end{aligned}\end{equation}
Take summation of the above $N$ equations and divide by $N$,
\begin{equation}\nonumber\begin{aligned}d\Big(\frac{1}{N}\sum_{i=1}^Nx_i(t)\Big)=&\Big[A(t)\frac{1}{N}\sum_{i=1}^Nx_i(t)-B(t)a(t)\frac{1}{N}\sum_{i=1}^N\mathbb{E}(x_i(t)|\mathcal{F}^{w_i}_t)+B(t)b(t)+\alpha x^{(N)}(t)+m(t)\\
&+B(t)\tilde{a}(t)\frac{1}{N}\sum_{i=1}^N\sum_{j=1,j\neq i}^{N}\mathbb{E}(x_j(t)|\mathcal{F}^{w_i}_t)\Big]dt+\sigma(t)\frac{1}{N}\sum_{i=1}^NdW_i(t)+\tilde{\sigma}(t)dW(t).\end{aligned}
\end{equation}Letting $N\rightarrow \infty$, we obtain the following limiting process which is a mean-field SDE:
\begin{equation}\label{X03}\left\{\begin{aligned}
dx_0(t)&=[(A(t)+\alpha)x_0-\tilde{\alpha}(t)\mathbb{E}x_0+\tilde{b}(t)]dt+\tilde{\sigma}(t)dW(t),\\
x_0(0)&=x
\end{aligned}\right.
\end{equation}where the functions $\tilde{\alpha}(\cdot), \tilde{b}(\cdot)$ are to be determined. Now, we introduce an auxiliary state as follows:
\begin{equation}\label{X3}\left\{\begin{aligned}
&dx_i(t)=[A(t)x_i(t)+B(t)u_i(t)+\alpha x_0(t)+m(t)]dt +\sigma(t)dW_i(t)+\tilde{\sigma}(t)dW(t),\\
&x_i(0)=x
\end{aligned}\right.
\end{equation}with the auxiliary cost functional
\begin{align}\label{J2}
J_i(u_i(\cdot))=\ &\mathbb{E}\left[\int_0^T \left(Q(t)(x_i(t)-x_0(t))^2
 + R(t) u^2_i(t)\right) dt+Gx^2_i(T)\right]
\end{align}where $x_0(\cdot)$ is given by \eqref{X03}. Thus, we formulate the following limiting partial filtration (\textbf{LPF}) LQG game. \\

\textbf{Problem (LPF).} For the $i^{th}$ agent, $i=1,2,\cdots,N,$ find $\bar{u}_i(\cdot)\in \mathcal{U}_i$ satisfying
$$J_i(\bar{u}_i(\cdot))=\inf_{u_i(\cdot)\in \mathcal{U}_i}J_i(u_i(\cdot)).$$
Then $\bar{u}_i(\cdot)$ is called an optimal control for Problem (\textbf{LPF}).\begin{remark}\emph{It is worth emphasizing the state process involved in (\textbf{LPF}) is given by \eqref{X3} which is different to state in (\textbf{PF}) which is given by \eqref{X1}. Specifically, the latter is affected by the actual state-average $x^{(N)}.$ Here, we still write them in the same notation to ease the presentation. }\end{remark}Applying the variational method, we have the following result to the optimal control of (\textbf{LPF}).
\begin{theorem}
Let \emph{(H1)} hold. Suppose there exists an optimal control $\bar{u}_i(\cdot)$ of Problem \emph{(\textbf{LPF})} and $\bar{x}_i(\cdot)$ is the corresponding optimal state, then there exists an adjoint process $p_i(\cdot)\in L_{\mathcal{F}^{i}}^2(0,T;\mathbb{R})$ satisfying the following backward stochastic differential equation (BSDE):
\begin{equation}\label{p}\left\{\begin{aligned}
dp_i(t)=&[-A(t)p_i(t)-Q(t)(\bar{x}_i(t)-x_0(t))]dt+\beta(t)dW_i(t)+\tilde{\beta}(t)dW(t),\\
p_i(T)=&G\bar{x}_i(T), \ i=1,2,\cdots,N
\end{aligned}\right.
\end{equation}such that$$\bar{u}_i(t) =-R^{-1}(t)B(t)\mathbb{E}(p_i(t)|\mathcal{F}^{w_i}_t)$$where the conditional expectation is defined in its optional projection version.
\end{theorem}Consequently, we get the following Hamiltonian system for $\mathcal{A}_i$:
\begin{equation}\label{Hs}\left\{\begin{aligned}
dx_0(t)=&[(A(t)+\alpha)x_0(t)-\tilde{\alpha}(t)\mathbb{E}x_0(t)+\tilde{b}(t)]dt+\tilde{\sigma}(t)dW(t),\\
d\bar{x}_i(t)=&[A(t)\bar{x}_i(t)-B^2(t)R^{-1}(t)\mathbb{E}(p_i(t)|\mathcal{F}^{w_i}_t)+\alpha x_0(t)+m(t)]dt\\
&+\sigma(t)dW_i(t)+\tilde{\sigma}(t)dW(t), \\
dp_i(t)=&[-A(t)p_i(t)-Q(t)(\bar{x}_i(t)-x_0(t))]dt+\beta(t)dW_i(t)+\tilde{\beta}(t)dW(t),\\
x_0(0)=&\bar{x}_i(0)=x,\ p_i(T)=G\bar{x}_i(T), \ i=1,2,\cdots, N.
\end{aligned}\right.
\end{equation}Note that in system \eqref{Hs}, the forward optimal state $\bar{x}_i(\cdot)$ depends on the backward adjoint process $p_i(\cdot)$ through its filtering state $\mathbb{E}(p_i(t)|\mathcal{F}^{w_i}_t)$. In this sense, \eqref{Hs} becomes a filtered FBSDE system and its decoupling should be proceeded through some FBSDE that involves the filtering state only. To this end, we introduce the following filter notations$$
\hat{\bar{x}}_i(t)=\mathbb{E}[\bar{x}_i(t)|\mathcal{F}^{w_i}_t],\ \ \ \hat{p}_i(t)=\mathbb{E}[p_i(t)|\mathcal{F}^{w_i}_t]
$$where the conditional expectations to the partial filtration $\mathcal{F}^{w_i}_t$ should be understood in the version of optional projection. Then we reach a FBSDE system
involving the state filters only:\begin{equation}\label{CHs}\left\{\begin{aligned}
d\hat{\bar{x}}_i(t)=&[A(t)\hat{\bar{x}}_i(t)-B^2(t)R^{-1}(t)\hat{p}_i(t)+\alpha \mathbb{E}x_0(t)+m(t)]dt+\sigma(t)dW_i(t), \\
\hat{\bar{x}}_i(0)=&x,\\
d\hat{p}_i(t)=&[-A(t)\hat{p}_i(t)-Q(t)(\hat{\bar{x}}_i(t)-\mathbb{E}x_0(t))]dt+\beta(t)dW_i(t),\\
\hat{p}_i(T)=&G\hat{\bar{x}}_i(T), \ i=1,2,\cdots, N.
\end{aligned}\right.
\end{equation}Note that system \eqref{CHs} is driven by $W_i$ only so it becomes observable to agent $\mathcal{A}_i.$ It can be viewed a filtering system of \eqref{Hs} that is unobservable as driven by $W_i$ and $W$ both. Taking expectation on \eqref{X03},
\begin{equation}\label{Ex}\left\{\begin{aligned}
d\mathbb{E}x_0(t)=&[(A(t)+\alpha-\tilde{\alpha}(t))\mathbb{E}x_0(t)+\tilde{b}(t)]dt,\\
\mathbb{E}x_0(0)=&x
\end{aligned}\right.
\end{equation}
where $\tilde{\alpha}(\cdot),\ \tilde{b}(\cdot)$ are functions to be determined. One key step in mean-field game is to analyze the related consistency condition (which is also called Nash certainty equivalence (NCE) principle, see \cite{gll}, \cite{hcm06}, \cite{hcm07}, etc).
Based on the consistency condition, the limiting state average $x_0$ can be determined in a consistent manner: the individual feedback close-loop states of all agents, when applying the derived decentralized strategies, should reformalize the limiting state average at the beginning. As $x_0$ follows \eqref{X03}, so its determination is equivalent to the determination of coefficient functions $(\tilde{\alpha}(\cdot),\ \tilde{b}(\cdot)).$ Thus we aim to derive the consistency condition satisfied by $(\tilde{\alpha}(\cdot),\ \tilde{b}(\cdot))$. We first state the following result.
\begin{theorem}\label{PP2}
Suppose \emph{(H1)} hold true and the following Riccati equation system
\begin{equation}\label{P}
\left \{
\begin{aligned}
&\dot{\Pi}(t)+(2A(t)+\alpha)\Pi(t)-B^2(t)R^{-1}(t)\Pi^2(t)=0,\\
&\dot{\Phi}(t)+[A(t)-B^2(t)R^{-1}(t)\Pi(t)]\Phi(t)+m(t)\Pi(t) =0,\\
&\Pi(T)=G,\ \Phi(T) =0
\end{aligned}
\right.
\end{equation}admits unique solution $(\Pi(\cdot),\Phi(\cdot))$, then $(\tilde{\alpha}(\cdot),\ \tilde{b}(\cdot))$ can be uniquely determined by\begin{equation}\label{ab}\left\{
\begin{aligned}
\tilde{\alpha}(t)&=B^2(t)R^{-1}(t)\Pi(t),\\
\tilde{b}(t)&=-B^2(t)R^{-1}(t)\Phi(t)+m(t).
\end{aligned}\right.
\end{equation}
\end{theorem}\begin{proof}By the terminal condition of \eqref{Hs}, we suppose
\begin{equation}\label{CP}
\hat{p}_i(t)=P(t)\hat{\bar{x}}_i(t)+\hat{P}(t)\mathbb{E}x_0(t)+\Phi(t)
\end{equation}
for some $P(\cdot),\hat{P}(\cdot)\in L^\infty (0,T;\mathbb{R})$ and $\Phi(t)\in L^\infty (0,T;\mathbb{R})$ with terminal conditions
$$
P(T)=G,\ \hat{P}(T)=\Phi(T)=0.
$$
Applying It\^{o}'s formula to \eqref{CP} and noting \eqref{Hs}, we have
\[%
\begin{array}
[c]{rl}%
d\hat{p}_i(t)
=&\Big(\dot{P}(t)+P(t)A(t)-B^2(t)R^{-1}(t)P^2(t)\Big)\hat{\bar{x}}_i(t)dt\\
&+\Big(\dot{\hat{P}}(t)+\hat{P}(t)(A(t)+\alpha-\tilde{\alpha}(t))-P(t)B^2(t)R^{-1}(t)\hat{P}(t)+\alpha P(t)\Big)\mathbb{E}x_0(t)dt\\
&+\Big(\dot{\Phi}(t)-P(t)B^2(t)R^{-1}(t)\Phi(t)+P(t)m(t)+\hat{P}(t)\tilde{b}(t)\Big)dt+P(t)\sigma(t)dW_i(t)\\
=&\Big[\left(-Q(t)-A(t)P(t)\right)\hat{\bar{x}}_i(t)+(Q(t)-A(t)\hat{P}(t))\mathbb{E}x_0(t)-A(t)\Phi(t)\Big]dt+\beta(t)dW_i(t).
\end{array}
\]
Comparing coefficients, we obtain
\begin{equation}\label{PPfai}
\left \{ \begin{aligned}
&\dot{P}(t)+P(t)A(t)-B^2(t)R^{-1}(t)P^2(t)=-Q(t)-A(t)P(t),\\
&\dot{\hat{P}}(t)+\hat{P}(t)(A(t)+\alpha-\tilde{\alpha}(t))-P(t)B^2(t)R^{-1}(t)\hat{P}(t)+\alpha P(t)=Q(t)-A(t)\hat{P}(t),\\
&\dot{\Phi}(t)-P(t)B^2(t)R^{-1}(t)\Phi(t)+P(t)m(t)+\hat{P}(t)\tilde{b}(t)=-A(t)\Phi(t),\\
&\beta(t)=P(t)\sigma(t).
\end{aligned} \right.
\end{equation}Note that the above Riccati equations are parameterized by the undetermined functions $(\tilde{\alpha}(t), \tilde{b}(t))$ which are to be specified below.
To this end, note that the optimal state $\bar{x}_i(t)$ can be represented by
\begin{equation}\nonumber
\begin{aligned}
d\bar{x}_i(t)=&[A(t)\bar{x}_i(t)-B^2(t)R^{-1}(t)(P(t)\hat{\bar{x}}_i(t)+\hat{P}(t)\mathbb{E}x_0(t)+\Phi(t))+\alpha x_0(t)+m(t)]dt\\
& +\sigma(t)dW_i(t)+\tilde{\sigma}(t)dW(t).
\end{aligned}
\end{equation}
Therefore the state-average satisfies:
\begin{equation}\nonumber
\begin{aligned}dx^{(N)}(t)=&\Big[A(t)x^{(N)}(t)-B^2(t)R^{-1}(t)(P(t)\frac{1}{N}\sum_{i=1}^N\mathbb{E}(\bar{x}_i(t)|\mathcal{F}^{w_i}_t)+\hat{P}(t)\\
&\cdot\mathbb{E}x_0(t)+\Phi(t))+\alpha x_0(t)+m(t)\Big]dt+\sigma(t)\frac{1}{N}\sum_{i=1}^NdW_i(t)+\tilde{\sigma}(t)dW(t).\end{aligned}
\end{equation}
Let $N \rightarrow +\infty,$ the limiting process $x_0$ is given by
\begin{equation}
\begin{aligned}
dx_0=&\Big[(A+\alpha)x_0-B^2R^{-1}(P+\hat{P})\mathbb{E}x_0-B^2R^{-1}\Phi+m\Big]dt+\tilde{\sigma}dW.
\end{aligned}
\end{equation}
Comparing the coefficients with \eqref{Hs}, we have
\begin{equation}\label{ab}\left\{
\begin{aligned}
\tilde{\alpha}(t)&=B^2(t)R^{-1}(t)(P(t)+\hat{P}(t)),\\
\tilde{b}(t)&=-B^2(t)R^{-1}(t)\Phi(t)+m(t).
\end{aligned}\right.
\end{equation}
Thus we rewrite \eqref{PPfai} as
\begin{equation}\nonumber
\left \{
\begin{aligned}
&\dot{P}(t)+2A(t)P(t)-B^2(t)R^{-1}(t)P^2(t)+Q(t)=0,\\
&\dot{\hat{P}}(t)+\hat{P}(t)[2A(t)+\alpha-B^2(t)R^{-1}(t)(P(t)+\hat{P}(t))-B^2(t)R^{-1}(t)P(t)]+\alpha P(t)-Q(t)=0,\\
&\dot{\Phi}(t)+[A(t)-(P(t)+\hat{P}(t))B^2(t)R^{-1}(t)]\Phi(t)+(P(t)+\hat{P}(t))m(t)=0,\\
&P(T)=G,\ \hat{P}(T)=\Phi(T) =0.
\end{aligned}
\right.
\end{equation}
Letting $\Pi(t)=P(t)+\hat{P}(t)$, we get
\begin{equation}\label{pai}\left\{
\begin{aligned}
&\dot{\Pi}(t)+(2A(t)+\alpha)\Pi(t)-B^2(t)R^{-1}(t)\Pi^2(t)=0,\\
&\Pi(T)=G.
\end{aligned}\right.
\end{equation}
\end{proof}Moreover,
the filtering system \eqref{CHs} can be decoupled as\begin{equation}\label{DHS}
\left \{
\begin{aligned}
&d\hat{\bar{x}}_i(t)=\Big[\Big(A(t)-B^2(t)R^{-1}(t)P(t)\Big)\hat{\bar{x}}_i(t)+\Big(\alpha-B^2(t)R^{-1}(t)(\Pi(t)-P(t))\Big) \mathbb{E}x_0(t)\\
&\qquad\qquad-B^2(t)R^{-1}(t)\Phi(t)+m(t)\Big]dt+\sigma(t)dW_i(t), \\
&\hat{p}_i(t)=P(t)\hat{\bar{x}}_i(t)+(\Pi(t)-P(t))\mathbb{E}x_0(t)+\Phi(t),\\
&\hat{\bar{x}}_i(0)=x_i(0),\  \hat{p}_i(T)=G\hat{\bar{x}}_i(T).
\end{aligned}
\right.
\end{equation}
Taking average of all and sending $N\rightarrow +\infty,$ we regenerate
\begin{equation}\label{DHS2}
\left \{
\begin{aligned}
d\mathbb{E}x_0=&[(A+\alpha-B^2R^{-1}\Pi)\mathbb{E}x_0-B^2R^{-1}\Phi+m]dt,\\
\mathbb{E}x_0(0)&=x.\end{aligned}
\right.
\end{equation}

\begin{remark}\ \emph{To conclude this section, we give some remarks concerning to Theorem {\ref{PP2}}.}

\emph{(1)}\emph{\quad Unlike most literature on mean-field LQG games, there has no fixed-point arguments explicitly involved here (e.g., some contraction mapping based on the datum of our problem) to characterize the consistency condition. Instead, our consistency condition is transformed into the wellposedness of Riccati equation system \eqref{P}. Actually, $(\hat{P}(\cdot), \Phi(\cdot))$ depend on $(\tilde{\alpha}(\cdot), \tilde{b}(\cdot)),$ thus \eqref{ab} can be rewritten by\begin{equation}\left\{
\begin{aligned}\nonumber
\tilde{\alpha}&=\mathcal{T}_1(\tilde{\alpha}):=B^2R^{-1}(P+\hat{P}(\tilde{\alpha})),\\
\tilde{b}&=\mathcal{T}_2(\tilde{b}):=-B^2R^{-1}\Phi(\tilde{\alpha}, \tilde{b})+m.
\end{aligned}\right.
\end{equation}In this sense, \eqref{P} can be understood as the consistency condition of (\textbf{LPF}) }.

\emph{(2)}\quad \emph{By {\cite{yong 2011},\ \cite{yz}}, it follows $P(\cdot)$ is determined uniquely as a nonnegative constant. One sufficient condition for the existence and uniqueness of $\Pi(\cdot)$ can be found in {\cite{my}} hence the solvability of $\hat{P}(\cdot)$ follows directly by noting $\Pi(t)=P(t)+\hat{P}(t)$. In addition, the solvability of $\Phi(\cdot)$ follows from that of $\Pi(\cdot).$}

\emph{(3)}\quad \emph{The advantages of handling the consistency condition of $(\tilde{\alpha}(\cdot), \tilde{b}(\cdot))$ are as follows. The consistency condition imposed on
$(\tilde{\alpha}(\cdot), \tilde{b}(\cdot))$ is equivalent to the wellposedness of Riccati equation \eqref{P} which can be ensured in an arbitrary time interval. On the other hand, as addressed in \cite{CD}, the fixed-point analysis on $x$ will preferably lead to the consistency condition defined only on a small time interval. This finding is also response to the standard result in forward-backward SDE theory: as discussed in \cite{Anto}, the usual contraction mapping on forward-backward system will always lead to its existence and uniqueness in a very small time interval. This should also be the case in mean-field games where the ``backward" Hamilton-Jacobi-Bellman (HJB) equation is coupled with the ``forward" Fokker-Planck (FP) equation.}
\end{remark}

\section{$\epsilon$-Nash Equilibrium for Problem (\textbf{PF})}
 Now we show that $(\bar{u}_1,\bar{u}_2,\cdots, \bar{u}_N)$ satisfies the $\epsilon$-Nash equilibrium for (\textbf{PF}).
\begin{definition}
A set of controls $u_k(\cdot)\in \mathcal{U}_k,\ 1\leq k\leq N,$ for $N$ agents is called an $\epsilon$-Nash equilibrium with respect to the costs $\mathcal{J}_k,\ 1\leq k\leq N,$ if there exists $\epsilon\geq0$ such that for any fixed $1\leq i\leq N$, we have
\begin{equation}\label{DNE}
\mathcal{J}_i(u_i,u_{-i})\leq\mathcal{J}_i(u'_i,u_{-i})+\epsilon
\end{equation}
when any alternative control $u'_i(\cdot)\in \mathcal{U}_i$ is applied by $\mathcal{A}_i$.
\end{definition}
\begin{theorem}\label{main}
Let \emph{(H1)} hold and \eqref{P} admit a solution $(\Pi, \Phi),$ then $(\bar{u}_1,\bar{u}_2,\cdots,\bar{u}_N)$ satisfies the $\epsilon$-Nash equilibrium of Problem \emph{(\textbf{PF})}. Here, for $1\le i\le N,$ $\bar{u}_i$ is given by\begin{equation}\label{ubari}
\bar{u}_i(t)=-R^{-1}(t)B(t)\Big[P(t)\hat{\bar{x}}_i(t)+(\Pi(t)-P(t))\mathbb{E}x_0(t)+\Phi(t)\Big]
\end{equation}
where $\hat{\bar{x}}_i$ and $\mathbb{E}x_0$ satisfy \eqref{DHS} and \eqref{DHS2} respectively.
\end{theorem}As preliminaries of proving the theorem, several lemmas are presented and proved to proceed some estimates on the state and cost difference between Problem (\textbf{PF}) and (\textbf{LPF}). Recall that
\begin{equation}\label{X4}\left\{\begin{aligned}
d\bar{x}_i(t)=&\Big[A(t)\bar{x}_i(t)-B^2(t)R^{-1}(t)\Big(P(t)\hat{\bar{x}}_i(t)+\hat{P}(t)\mathbb{E}x_0(t)+\Phi(t)\Big)+\alpha x_0(t)+m(t)\Big]dt\\
& +\sigma(t)dW_i(t)+\tilde{\sigma}(t)dW(t),\\
d\hat{\bar{x}}_i(t)=&\Big[\Big(A(t)-B^2(t)R^{-1}(t)P(t)\Big)\hat{\bar{x}}_i(t)+\Big(\alpha-B^2(t)R^{-1}(t)\hat{P}(t)\Big)\mathbb{E}x_0(t)\\
&-B^2(t)R^{-1}(t)\Phi(t)+m(t)\Big]dt+\sigma(t)dW_i(t),\\
\bar{x}_i(0)=&\hat{\bar{x}}_i(t)=x_i(0),
\end{aligned}\right.
\end{equation}and denote$$
\bar{x}^{(N)}(t)=\frac{1}{N}\sum_{i=1}^N\bar{x}_i(t),\ \ \ \ \hat{\bar{x}}^{(N)}(t)=\frac{1}{N}\sum_{i=1}^N\hat{\bar{x}}_i(t).
$$
Here, $\bar{x}^{(N)}(t)$ denotes the average of state (in Problem (LPF)) while $\hat{\bar{x}}^{(N)}$ denotes the average of filtered states. Note that $\hat{\bar{x}}_i(t)$ is driven by $W_i$ only thus it is observable to the individual agent $\mathcal{A}_i.$ It enters the state dynamics in \eqref{X4} as an input process when applying the optimal strategy constructed on it. Some estimates are as follows.\begin{lemma}\label{xNhat}
\begin{flalign}
&\sup_{0\leq t\leq T}\mathbb{E}\Big|\hat{\bar{x}}^{(N)}(t)-\mathbb{E}x_0(t)\Big|^2=O\Big(\frac{1}{N}\Big),\label{e1}\\
&\sup_{0\leq t\leq T}\mathbb{E}\Big|\bar{x}^{(N)}(t)-x_0(t)\Big|^2=O\Big(\frac{1}{N}\Big).\label{e2}
\end{flalign}
\end{lemma}

\begin{proof} See Appendix A.
\end{proof}Denote $y_i,1\le i\le N,$ the state of $\mathcal{A}_i$ to the control $\bar{u}_i,1\le i\le N$ in Problem (\textbf{PI}), namely,
\begin{equation}\label{y}\left\{\begin{aligned}
dy_i(t)=&\Big[A(t)y_i(t)-B^2(t)R^{-1}(t)\Big(P(t)\hat{\bar{x}}_i(t)+\hat{P}(t)\mathbb{E}x_0(t)+\Phi(t)\Big)+\alpha y^{(N)}(t)\\
&+m(t)\Big]dt+\sigma(t)dW_i(t)+\tilde{\sigma}(t)dW(t),\\
y_i(0)=&x_i(0)
\end{aligned}\right.
\end{equation}where $y^{(N)}(t)=\frac{1}{N}\sum_{j=1}^Ny_j(t)$.

By the difference of states corresponding to $\bar{u}_i$ in (\textbf{PI}) and (\textbf{LPI}), we have the following estimates:
\begin{lemma}\label{yN}
\begin{flalign}
&\sup_{0\leq t\leq T}\mathbb{E}\Big|y^{(N)}(t)-x_0(t)\Big|^2=O\Big(\frac{1}{N}\Big),\label{e3}\\
&\sup_{1\leq i\leq N}\left[\sup_{0\leq t\leq T}\mathbb{E}\Big|y_i(t)-\bar{x}_i(t)\Big|^2\right]=O\Big(\frac{1}{N}\Big),\label{e4}\\
&\sup_{1\leq i\leq N}\left[\sup_{0\leq t\leq T}\mathbb{E}\Big||y_i(t)|^2-|\bar{x}_i(t)|^2\Big|\right]=O\Big(\frac{1}{\sqrt{N}}\Big).\label{e5}
\end{flalign}
\end{lemma}
\begin{proof}By \eqref{y} and \eqref{DHS}, the estimate \eqref{e3} can be verified by the same method in Lemma \ref{xNhat}. According to \eqref{y} and \eqref{X4}, we have
\begin{equation}\nonumber\left\{\begin{aligned}
&d\Big(y_i(t)-\bar{x}_i(t)\Big)=\Big[A(t)(y_i(t)-\bar{x}_i(t))+\alpha(y^{(N)}(t)-x_0(t))\Big]dt,\\
&y_i(0)-\bar{x}_i(0)=0.
\end{aligned}\right.
\end{equation}Thus, \eqref{e4} follows from \eqref{e3}. Since $\sup_{0\leq t\leq T}\mathbb{E}|\bar{x}_i(t)|^2<+\infty$, applying Cauchy-Schwarz inequality, we get
\begin{equation}\nonumber\begin{aligned}
&\sup_{0\leq t\leq T}\mathbb{E}\Big||y_i(t)|^2-|\bar{x}_i(t)|^2\Big|\\
\leq& \sup_{0\leq t\leq T}\mathbb{E}|y_i(t)-\bar{x}_i(t)|^2+2\sup_{0\leq t\leq T}\mathbb{E}|\bar{x}_i(t)(y_i(t)-\bar{x}_i(t))|\\
\leq&\sup_{0\leq t\leq T}\mathbb{E}|y_i(t)-\bar{x}_i(t)|^2+2\Big(\sup_{0\leq t\leq T}\mathbb{E}|\bar{x}_i(t)|^2\Big)^{\frac{1}{2}}\Big(\sup_{0\leq t\leq T}\mathbb{E}|y_i(t)-\bar{x}_i(t)|^2\Big)^{\frac{1}{2}}\\
=&O\Big(\frac{1}{\sqrt{N}}\Big)
\end{aligned}
\end{equation}which completes the proof.
\end{proof}As to the difference of cost functionals, it holds
\begin{lemma}\label{l8}For $\forall\ 1\leq i\leq N,$
\begin{equation}\label{e6}
 \Big|\mathcal{J}_i(\bar{u}_i,\bar{u}_{-i})-J_i(\bar{u}_i)\Big|=O\Big(\frac{1}{\sqrt{N}}\Big).
\end{equation}
\end{lemma}
\begin{proof} See Appendix B.
\end{proof}After addressing the above estimates of states and costs corresponding to control $\bar{u}_i,1\le i\le N$, given by \eqref{ubari}, our goal is to prove that the control strategies set $(\bar{u}_1,\cdots,\bar{u}_N)$ is an $\epsilon$-Nash equilibrium for Problem (\textbf{PI}). For any fixed $i$, $1\le i\le N$, consider an admissible control $u_i \in \mathcal{U}_i$ for $\mathcal{A}_i$ and denote $z_i$ the corresponding state process in Problem (\textbf{PI}), that is\begin{equation}\label{y2}\left\{\begin{aligned}
dz_i(t)=&\Big[A(t)z_i(t)+B(t)u_i(t)+\alpha z^{(N)}(t)+m(t)\Big]dt +\sigma(t)dW_i(t)+\tilde{\sigma}(t)dW(t),\\
z_i(0)=&x_i(0)
\end{aligned}\right.
\end{equation}whereas other agents keep the control $\bar{u}_j, 1\le j\le N,j\neq i$, i.e.,
\begin{equation}\label{y3}\left\{\begin{aligned}
dz_j(t)=&\Big[A(t)z_j(t)-B^2(t)R^{-1}(t)\Big(P(t)\hat{\bar{x}}_j(t)+\hat{P}(t)\mathbb{E}x_0(t)+\Phi(t)\Big)\\
&+\alpha z^{(N)}(t)+m(t)\Big]dt+\sigma(t)dW_j(t)+\tilde{\sigma}(t)dW(t),\\
z_j(0)=&x_j(0)
\end{aligned}\right.
\end{equation}
where $z^{(N)}(t)=\frac{1}{N}\sum_{j=1}^Nz_j(t)$ and $\hat{\bar{x}}_j(t)$ is given by \eqref{X4}. If $\bar{u}_i,\ 1\leq i\leq N$ is an $\epsilon$-Nash equilibrium with respect to cost $\mathcal{J}_i$, it holds that$$\mathcal{J}_i(\bar{u}_i,\bar{u}_{-i})\geq \inf_{u_i\in \mathcal{U}_i}\mathcal{J}_i(u_i,\bar{u}_{-i})\geq \mathcal{J}_i(\bar{u}_i,\bar{u}_{-i})-\epsilon.$$
Then, when making the perturbation, we just need to consider $u_i\in \mathcal{U}_i$ such that $\mathcal{J}_i(u_i,\bar{u}_{-i})\leq \mathcal{J}_i(\bar{u}_i,\bar{u}_{-i}),$ which implies
\begin{equation}\nonumber\begin{aligned}
\mathbb{E}\int_0^TR(t)u^2_i(t)dt\leq \mathcal{J}_i(u_i,\bar{u}_{-i})\leq \mathcal{J}_i(\bar{u}_i,\bar{u}_{-i})=J_i(\bar{u}_i)+O\Big(\frac{1}{\sqrt{N}}\Big).
\end{aligned}
\end{equation}
In the limiting cost functional, by the optimality of $(\bar{x}_i,\bar{u}_i)$, we get that $(\bar{x}_i,\bar{u}_i)$ is $L^2$-bounded. Then we obtain the boundedness of $J_i(\bar{u}_i)$, i.e.,
\begin{equation}\label{ubound}
    \mathbb{E}\int_0^TR(t)u^2_i(t)dt\leq C_0
\end{equation}where $C_0$ is a positive constant, independent of $N$. Thus we have
\begin{prop}\label{z111}
For any fixed $i,1\le i\le N$, $\sup\limits_{0\leq t\leq T}\mathbb{E}|z_i(t)|^2$ is bounded.
\end{prop}
\begin{proof} See Appendix C.
\end{proof}Correspondingly, the state process $\bar{x}^0_i$ for agent $\mathcal{A}_i$ under control $u_i$ in Problem (\textbf{LPI}) satisfies


\begin{equation}\label{y4}\left\{\begin{aligned}
d\bar{x}^0_i(t)=&\Big[A(t)\bar{x}^0_i(t)+B(t)u_i(t)+\alpha x_0(t)+m(t)\Big]dt +\sigma(t)dW_i(t)+\tilde{\sigma}(t)dW(t),\\
\bar{x}^0_i(0)=&x_i(0)
\end{aligned}\right.
\end{equation}
and for agent $\mathcal{A}_j$, $j\neq i$,
\begin{equation}\label{y5}\left\{\begin{aligned}
d\bar{x}_j(t)=&\Big[A(t)\bar{x}_j(t)-B^2(t)R^{-1}(t)\Big(P(t)\hat{\bar{x}}_j(t)+\hat{P}(t)\mathbb{E}x_0(t)+\Phi(t)\Big)\\
&+\alpha x_0(t)+m(t)\Big]dt+\sigma(t)dW_j(t)+\tilde{\sigma}(t)dW(t),\\
\bar{x}_j(0)=&x_j(0)
\end{aligned}\right.
\end{equation}
where $\hat{\bar{x}}_j$ and $x_0$ are given in \eqref{DHS}.

In order to give necessary estimates of perturbed states and costs in Problem (\textbf{PI}) and (\textbf{LPI}), we introduce some intermediate states and present some of their properties.
Denote $$z^{(N-1)}(t)=\frac{1}{N-1}\sum_{j=1,j\neq i}^Nz_j(t),\ \ \hat{\bar{x}}^{(N-1)}(t)=\frac{1}{N-1}\sum_{j=1,j\neq i}^N\hat{\bar{x}}_j(t).$$
Then by \eqref{y3}, we have
\begin{equation}\label{zN}\left\{\begin{aligned}
dz^{(N-1)}(t)=&\Big[(A(t)+\frac{N-1}{N}\alpha)z^{(N-1)}(t)-B^2(t)R^{-1}(t)\Big(P(t)\hat{\bar{x}}^{(N-1)}(t)+\hat{P}(t)\mathbb{E}x_0(t)\\
&+\Phi(t)\Big)+\frac{\alpha}{N}z_i(t)+m(t)\Big]dt+\frac{1}{N-1}\sum_{j=1,j\neq i}^N\sigma(t)dW_j(t)+\tilde{\sigma}(t)dW(t),\\
z^{(N-1)}(0)=&x^{(N-1)}(0)
\end{aligned}\right.
\end{equation}
where $x^{(N-1)}(0)=\frac{1}{N-1}\sum_{j=1,j\neq i}^Nx_j(0)$. Besides, we introduce
\begin{equation}\label{y6}\left\{\begin{aligned}
d\check{z}_i(t)=&\Big[A(t)\check{z}_i(t)+B(t)u_i(t)+\frac{N-1}{N}\alpha\check{z}^{(N-1)}(t)+m(t)\Big]dt\\
&+\sigma(t)dW_i(t)+\tilde{\sigma}(t)dW(t),\\
\check{z}_i(0)=&x_i(0)
\end{aligned}\right.
\end{equation}and for $j\neq i$,\begin{equation}\label{y7}\left\{\begin{aligned}
d\check{z}_j(t)=&\Big[A(t)\check{z}_j(t)-B^2(t)R^{-1}(t)\Big(P(t)\hat{\bar{x}}_j(t)+\hat{P}(t)\mathbb{E}x_0(t)+\Phi(t)\Big)\\
&+\frac{N-1}{N}\alpha \check{z}^{(N-1)}(t)+m(t)\Big]dt+\sigma(t)dW_j(t)+\tilde{\sigma}(t)dW(t),\\
\check{z}_j(0)=&x_j(0)
\end{aligned}\right.
\end{equation}where $\check{z}^{(N-1)}(t)=\frac{1}{N-1}\sum_{j=1,j\neq i}^N\check{z}_j(t)$.

We have the following estimates on these states.
\begin{prop}\label{allz}
\begin{align}\label{zzN}
&\sup_{0\leq t\leq T}\mathbb{E}\Big|\hat{\bar{x}}^{(N-1)}(t)-\mathbb{E}x_0(t)\Big|^2=O\Big(\frac{1}{N}\Big),\\
\label{zN1x0}
&\sup_{0\leq t\leq T}\mathbb{E}\Big|z^{(N)}(t)-z^{(N-1)}(t)\Big|^2=O\Big(\frac{1}{N}\Big),\\
\label{zizi}
&\sup_{0\leq t\leq T}\mathbb{E}\Big|\check{z}^{(N-1)}(t)-z^{(N-1)}(t)\Big|^2=O\Big(\frac{1}{N^2}\Big),\\
\label{yN1Ex0}
&\sup_{0\leq t\leq T}\mathbb{E}\Big|\check{z}^{(N-1)}(t)-x_0(t)\Big|^2=O\Big(\frac{1}{N}\Big),\\
\label{zzN1}
&\sup_{0\leq t\leq T}\mathbb{E}\Big|z_i(t)-\check{z}_i(t)\Big|^2=O\Big(\frac{1}{N^2}\Big),\\
\label{zxi0}
&\sup_{0\leq t\leq T}\mathbb{E}\Big|\check{z}_i(t)-\bar{x}^0_i(t)\Big|^2=O\Big(\frac{1}{N}\Big).
\end{align}
\end{prop}
\begin{proof}
See Appendix D.
\end{proof}

Further, more direct estimates about states and costs of Problem (\textbf{PI}) and (\textbf{LPI}) under perturbed controls can be obtained, which enable us to prove Theorem \ref{main}.


\begin{lemma}\label{p3}
\begin{align}\label{Le1}
&\sup_{0\leq t\leq T}\mathbb{E}\Big|z_i(t)-\bar{x}^0_i(t)\Big|^2=O\Big(\frac{1}{N}\Big),\\ \label{Le2}
&\sup_{0\leq t\leq T}\mathbb{E}\Big|z^{(N)}(t)-x_0(t)\Big|^2=O\Big(\frac{1}{N}\Big),\\ \label{Le3}
&\sup_{0\leq t\leq T}\mathbb{E}\Big||z_i(t)|^2-|\bar{x}^0_i(t)|^2\Big|=O\Big(\frac{1}{\sqrt{N}}\Big),\\  \label{Le4}
&\Big|\mathcal{J}_i(u_i,\bar{u}_{-i})-J_i(u_i)\Big|=O\Big(\frac{1}{\sqrt{N}}\Big).
\end{align}
\end{lemma}
\begin{proof}
See Appendix E.
\end{proof}

\emph{Proof of Theorem \ref{main}:}
 Consider the $\epsilon$-Nash equilibrium for $\mathcal{A}_i$. Combining Lemma \ref{l8} and \ref{p3}, we have
\begin{equation}\nonumber\begin{aligned}
\mathcal{J}_i(\bar{u}_i,\bar{u}_{-i})&=J_i(\bar{u}_i)+O\Big(\frac{1}{\sqrt{N}}\Big)\\
&\leq J_i(u_i)+O\Big(\frac{1}{\sqrt{N}}\Big)\\
&=\mathcal{J}_i(u_i,\bar{u}_{-i})+O\Big(\frac{1}{\sqrt{N}}\Big).
\end{aligned}
\end{equation}
Thus, Theorem \ref{main} follows by taking $\epsilon=O\Big(\frac{1}{\sqrt{N}}\Big)$.

\section{Mean-Field LQG (MFLQG) Games with Partial Observation}

We consider a large population system with noisy observation as follows. The state $x_i$ for each $\mathcal{A}_{i}$ still satisfies the following linear stochastic system:
\begin{equation}\label{o1}
dx_i(t)=[A(t)x_i(t)+B(t)u_i(t)+\alpha x^{(N)}(t)+m(t)]dt+\sigma(t)dW_i(t)+\tilde{\sigma}(t)dW(t), \quad x_i(0)=x
\end{equation}
with $x^{(N)}(t)=\frac{1}{N}\sum_{i=1}^{N}x_{i}(t).$ We still assume the initial states $x_i(0)$ are deterministic and same for all agents. The $i^{th}$ agent $\mathcal{A}_i$ can access the following additive white-noise partial observation:\begin{equation}\label{o2}
dy_i(t)=[H(t)x_i(t)+\tilde{H}(t) x^{(N)}(t)+h(t)]dt+dV_{i}(t).
\end{equation}Here, $\{V_i\}_{1 \leq i \leq N}, \{W_i\}_{1 \leq i \leq N}, W$ are independent Brownian motions. Define the observable filtration of $\mathcal{A}_i$ as $\mathcal{G}^i_t:=\sigma\{y_i(s),W(s);0\leq s\leq t\}.$ In what follows, for simplicity, we focus on the one-dimensional case. The admissible control $u_i\in \mathcal{U}_i$ where the admissible control set $\mathcal{U}_i$ is defined as
$$
\mathcal{U}_i:=\{u_i(\cdot)|u_i(\cdot)\in L^{2}_{\mathcal{G}_t^{i}}(0, T; \mathbb{R})\},\ 1\leq i \leq N.
$$Let $u=(u_1, \cdots u_i, \cdots, u_{N})$ still denote the set of control strategies of all $N$ agents; $u_{-i}=(u_1, \cdots, u_{i-1},$ $u_{i+1}, \cdots u_{N})$ the control strategies set except the $i^{th}$ agent $\mathcal{A}_i.$ Introduce the cost functional of $\mathcal{A}_i$ as
\begin{align}\label{o3}
\mathcal{J}_i(u_i(\cdot),u_{-i}(\cdot))=\ &\mathbb{E}\left[\int_0^T \left(Q(t)\big(x_i(t)-x^{(N)}(t)\big)^2
 + R(t) u_i^2(t)\right) dt+G x_i^2(T)\right].
\end{align}For the coefficients of (\ref{o1}), \eqref{o2} and (\ref{o3}), we set the following assumption:
\begin{equation*}
(H2)\qquad
\begin{aligned}
H(\cdot),\tilde{H}(\cdot),h(\cdot)\in L^\infty(0,T;\mathbb{R}).
\end{aligned}
\end{equation*}Now, we formulate the large population LQG games with partial observation (\textbf{PO}).\\

\textbf{Problem (PO).}
Find a control strategies set $\bar{u}=(\bar{u}_1,\bar{u}_2,\cdots,\bar{u}_N)$ which satisfies
$$
\mathcal{J}_i(\bar{u}_i(\cdot),\bar{u}_{-i}(\cdot))=\inf_{u_i(\cdot)\in \mathcal{U}_i}\mathcal{J}_i(u_i(\cdot),\bar{u}_{-i}(\cdot))
$$where $\bar{u}_{-i}$ represents $(\bar{u}_1,\cdots,\bar{u}_{i-1},\bar{u}_{i+1},\cdots, \bar{u}_N)$.\

To study (\textbf{PO}) and obtain the decentralized individual strategies, a key step is to approximate the state-average. To this end, we define the state filter for the given filtration $\mathcal{G}^i_t$ as $$\hat{x}_i(t)=\mathbb{E}(x_i(t)|\mathcal{G}^i_t).$$ Then $\hat{x}^{(N)}(t)=\frac{1}{N}\sum_{i=1}^N\hat{x}_i(t)$ denotes the average of state filters. This is very different to the full-information analysis of large-population system in which we need only study the average of state itself. Applying the similar analysis as in Section 2, $x^{(N)}$ and $\hat{x}^{(N)}$ are approximated by $x_0$ and $\hat{x}$ respectively which are given by
\begin{equation}\label{o4}\left\{\begin{aligned}
dx_0(t)&=[\tilde{A}(t)x_0(t)+\tilde{B}(t)\hat{x}(t)+f(t)]dt+\tilde{\sigma}(t)dW(t),\\
d\hat{x}(t)&=[C(t)x_0(t)+D(t)\hat{x}(t)+g(t)]dt,\\
x_0(0)&=\hat{x}(0)=x
\end{aligned}\right.
\end{equation}
where $(\tilde{A}(\cdot),\tilde{B}(\cdot),C(\cdot),D(\cdot))\in L^2(0,T;\mathbb{R})$ and $(f(\cdot),g(\cdot))\in L^2_{\mathcal{G}_t^w}(0,T;\mathbb{R})$ are to be determined. It follows that the individual state filters in our partial observation case are no longer independent. This is different to the previous partial filtration case in Section 2, 3 where the state filters are independent thus only its expectation is needed. Therefore, here we need average all state filters by considering their conditional independence on filtration $\{\mathcal{G}^w_t\}_{0 \leq t \leq T}.$ Now, we introduce the limiting partial-observation state and limiting observation as follows:
\begin{equation}\label{o5}\left\{\begin{aligned}
&dx_i(t)=[A(t)x_i(t)+B(t)u_i(t)+\alpha x_0(t)+m(t)]dt +\sigma(t)dW_i(t)+\tilde{\sigma}(t)dW(t),\\
&x_i(0)=x
\end{aligned}\right.
\end{equation}
and \begin{equation}\label{o6}\begin{aligned}
&d\bar{y}_i(t)=[H(t)x_i(t)+\tilde{H}(t) x_0(t)+h(t)]dt+dV_{i}(t).
\end{aligned}
\end{equation}Note that the auxiliary state $x_i(\cdot)$ is different from that in \eqref{o1} although they apply the same notations. The limiting cost functional is given by
\begin{align}\label{o7}
J_i(u_i(\cdot))=\ &\mathbb{E}\left[\int_0^T \left(Q(t)\big(x_i(t)-x_0(t)\big)^2
 + R(t) u_i^2(t)\right) dt+G x_i^2(T)\right]
\end{align}where $x_0(\cdot)$ is given by \eqref{o4}. Now we formulate the following limiting partial observation (\textbf{LPO}) LQG game. \\

\textbf{Problem (LPO).} For the $i^{th}$ agent, $i=1,2,\cdots,N,$ find $\bar{u}_i(\cdot)\in \mathcal{U}_i$ satisfying
$$J_i(\bar{u}_i(\cdot))=\inf_{u_i(\cdot)\in \mathcal{U}_i}J_i(u_i(\cdot)).$$
Then $\bar{u}_i(\cdot)$ is called an optimal control for Problem (\textbf{LPO}).

As to the optimal control with partial observation (see \cite{Ben92}), we have
\begin{equation}\label{o14}\left\{\begin{aligned}
d\hat{x}_{i}(t)=&\big[A(t)\hat{x}_{i}(t)+B(t)u_i(t)+\alpha x_0(t)+m(t)\big]dt+P(t)H(t)dW(t)\\
&+P(t)H(t)\big[d\bar{y}_{i}(t)-(H(t)\hat{x}_i(t)+\tilde{H}(t) x_0(t)+h(t))dt\big],\\
\hat{x}_{i}(0)=&x
\end{aligned}\right.
\end{equation}
where $P(\cdot)\in L^{2}(0,T;\mathbb{R})$ is the unique solution of the Riccati equation
\begin{equation}\label{o15}\left\{\begin{aligned}
\dot{P}(t)&=2A(t)P(t)-H^2(t)P^2(t)+\sigma^2(t)+\tilde{\sigma}^2(t),\\
P(0)&=0.
\end{aligned}\right.
\end{equation}Now, we introduce some notations:
\begin{equation}\nonumber\begin{aligned}
&X_{i}(\cdot)=\left(
          \begin{array}{c}
            x_{i}(\cdot) \\
            x_0(\cdot) \\
            \hat{x}(\cdot) \\
          \end{array}
        \right),\hat{X}_{i}(\cdot)=\left(
          \begin{array}{c}
            \hat{x}_{i}(\cdot) \\
            x_0(\cdot) \\
            \hat{x}(\cdot) \\
          \end{array}
        \right),
        \mathbb{A}(\cdot)=\left(
                     \begin{array}{ccc}
                       A(\cdot) & \alpha & 0 \\
                       0 & \tilde{A}(\cdot) & \tilde{B}(\cdot) \\
                       0 & C(\cdot) & D(\cdot) \\
                     \end{array}
                   \right),\\
                   &\mathbb{B}(\cdot)=\left(
                                               \begin{array}{c}
                                                 B(\cdot) \\
                                                 0 \\
                                                 0 \\
                                               \end{array}
                                             \right),              \quad      \mathbb{F}(\cdot)=\left(
          \begin{array}{c}
            m(\cdot) \\
            f(\cdot) \\
            g(\cdot) \\
          \end{array}
        \right),\ \ \Sigma(\cdot)=\left(
                         \begin{array}{cc}
                           P(\cdot)H(\cdot) & P(\cdot)H(\cdot) \\
                           0 & \tilde{\sigma}(\cdot) \\
                           0 & 0 \\
                         \end{array}
                       \right),\\
         &\mathbb{Q}(\cdot)=\left(
                                                                    \begin{array}{ccc}
                                                                      Q(\cdot) & -Q(\cdot) & 0 \\
                                                                      -Q(\cdot) & Q(\cdot) & 0 \\
                                                                      0 & 0 & 0 \\
                                                                    \end{array}
                                                                  \right),                                                                  \mathbb{G }=\left(
                                                                                  \begin{array}{ccc}
                                                                                    G & 0 & 0 \\
                                                                                    0 & 0 & 0 \\
                                                                                    0 & 0 & 0 \\
                                                                                  \end{array}
                                                                                \right)
                                                                  , \mathbb{I}=\left(
                                                                                     \begin{array}{c}
                                                                                       1 \\
                                                                                       1 \\
                                                                                       1 \\
                                                                                     \end{array}
                                                                                   \right).
\end{aligned}
\end{equation}
Introduce the innovation process
\begin{equation}\nonumber\begin{aligned}
I_i(t):=\bar{y}_{i}(t)-\int_0^t[H(s)\hat{x}_i(s)+\tilde{H}(s) x_0(s)+h(s)]ds.
\end{aligned}
\end{equation}
Then the system \eqref{o4}, \eqref{o5} and \eqref{o7} can be rewritten as
\begin{equation}\label{o151}\begin{aligned}
d\hat{X}_{i}(t)=&\big[\mathbb{A}(t)\hat{X}_{i}(t)+\mathbb{B}(t)u_i(t)+\mathbb{F}(t)\big]dt+\Sigma(t)\left(
                                                                                                      \begin{array}{c}
                                                                                                        dI_i(t) \\
                                                                                                        dW(t) \\
                                                                                                      \end{array}
                                                                                                    \right)
,\ \hat{X}_{i}(0)=x\mathbb{I}
\end{aligned}
\end{equation}
with
\begin{equation}\label{o152}\begin{aligned}
J_i(u_i(\cdot))=\ \mathbb{E}\left[\int_0^T \left(X'_{i}(t)\mathbb{Q}(t)X_{i}(t) + R(t) u_i^2(t)\right) dt+X'_{i}(T)\mathbb{G}X_{i}(T)\right].
\end{aligned}
\end{equation}
Given two functions $\pi(\cdot)\in L^\infty(0,T; \mathbb{R}^{3\times3})$ and $\gamma(\cdot)\in L^\infty_{\mathcal{G}^w_t}(0,T; \mathbb{R}^{3\times1})$ satisfying
\begin{equation}\label{o16}\left\{\begin{aligned}
&\dot{\pi}(t)+\pi(t)\mathbb{A}(t)+\mathbb{A}'(t)\pi(t)-R^{-1}(t)\pi(t)\mathbb{B}(t)\mathbb{B}'(t)\pi(t)+\mathbb{Q}(t)=0,\\
&\pi(T)=\mathbb{G}
\end{aligned}\right.
\end{equation}
and
\begin{equation}\label{o17}\left\{\begin{aligned}
&d\gamma(t)+[(\mathbb{A}'(t)-R^{-1}(t)\pi(t)\mathbb{B}(t)\mathbb{B}'(t))\gamma(t)+\pi(t)\mathbb{F}(t)]dt+q(t)dW(t)=0,\\
&\gamma(T)=0
\end{aligned}\right.
\end{equation}
respectively. We have
\begin{lemma}
Let \emph{(H1)-(H2)} hold. Assume \eqref{o16} and \eqref{o17} admit unique solutions $\pi(\cdot)$ and $(\gamma(\cdot),q(\cdot))$, then if the optimal control of Problem \emph{(\textbf{LPO})} exist, it can be represented by $$\bar{u}_i(t)=-R^{-1}(t)\mathbb{B}'(t)\pi(t)\hat{X}_{i}(t)-R^{-1}(t)\mathbb{B}'(t)\gamma(t).$$
\end{lemma}
The proof can be seen in \cite{Ben92} and omitted. Now, we aim to derive the consistency condition satisfied by the decentralized control policy. From \eqref{o16}, we get that $\pi(\cdot)$ is symmetric and denote $(\pi(\cdot),\gamma(\cdot))$ as $$\pi(\cdot)=\left(
                              \begin{array}{ccc}
                                \pi_{11}(\cdot) & \pi_{12}(\cdot) & \pi_{13}(\cdot) \\
                                \pi_{12}(\cdot) & \pi_{22}(\cdot) & \pi_{23}(\cdot) \\
                                \pi_{13}(\cdot) & \pi_{23}(\cdot) & \pi_{33}(\cdot) \\
                              \end{array}
                            \right),\qquad
\gamma(\cdot)=\left(
                                                 \begin{array}{c}
                                                   \gamma_1(\cdot) \\
                                                   \gamma_2(\cdot) \\
                                                   \gamma_3(\cdot) \\
                                                 \end{array}
                                               \right).
$$
Then the optimal control $\bar{u}_i(t)$ is rewritten as
\begin{equation}\label{o18}\begin{aligned}
\bar{u}_i(t)=-R^{-1}(t)B(t)\big(\pi_{11}(t)\hat{x}_i(t)+\pi_{12}(t)x_0(t)+\pi_{13}\hat{x}(t)+\gamma_1(t)\big).
\end{aligned}
\end{equation}
Plugging \eqref{o18} into \eqref{o1}, we have
\begin{equation}\nonumber
\begin{aligned}
dx_i(t)=&[A(t)x_i(t)-B^2(t)R^{-1}(t)\big(\pi_{11}(t)\hat{x}_i(t)+\pi_{12}(t)x_0(t)+\pi_{13}\hat{x}(t)+\gamma_1(t)\big)\\
&+\alpha x^{(N)}(t)+m(t)]dt +\sigma(t)dW_i(t)+\tilde{\sigma}(t)dW(t).
\end{aligned}
\end{equation}
Taking summation, dividing by $N$ and letting $N\rightarrow+\infty$, we obtain
\begin{equation}\label{o181}
\begin{aligned}
dx_0(t)=&\Big[\big(A(t)+\alpha-B^2(t)R^{-1}(t)\pi_{12}(t)\big)x_0(t)-B^2(t)R^{-1}(t)\big(\pi_{11}(t)+\pi_{13}(t)\big)\hat{x}(t)\\
&-B^2(t)R^{-1}(t)\gamma_1(t)+m(t)\Big]+\tilde{\sigma}(t)dW(t).
\end{aligned}
\end{equation}
Comparing the coefficients with \eqref{o4}, we have
\begin{equation}\label{o19}\left\{\begin{aligned}
&\tilde{A}(t)=A(t)+\alpha-B^2(t)R^{-1}(t)\pi_{12}(t),\\
&\tilde{B}(t)=-B^2(t)R^{-1}(t)\big(\pi_{11}(t)+\pi_{13}(t)\big),\\
&f(t)=-B^2(t)R^{-1}(t)\gamma_1(t)+m(t).
\end{aligned}\right.
\end{equation}
From \eqref{o14} and \eqref{o18} it follows that
\begin{equation}\nonumber\begin{aligned}
d\hat{x}_{i}(t)=&\Big[A(t)\hat{x}_{i}(t)-B^2(t)R^{-1}(t)\big(\pi_{11}(t)\hat{x}_i(t)+\pi_{12}(t)x_0(t)+\pi_{13}\hat{x}(t)+\gamma_1(t)\big)+\alpha x_0(t)+m(t)\Big]dt\\
&+P(t)H(t)dW(t)+P(t)H(t)\Big[\big(H(t)x_i(t)-H(t)\hat{x}_i(t)\big)dt+dV_i(t)\Big].
\end{aligned}
\end{equation}
Taking summation, dividing by $N$ and letting $N\rightarrow+\infty$, we get
\begin{equation}\label{o20}\begin{aligned}
d\hat{x}(t)=&\Big[\Big(A(t)-B^2(t)R^{-1}(t)\big(\pi_{11}(t)+\pi_{13}(t)\big)-P(t)H^2(t)\Big)\hat{x}(t)-B^2(t)R^{-1}(t)\gamma_1(t)+m(t)\\
&+\Big(\alpha-B^2(t)R^{-1}(t)\pi_{12}(t)+P(t)H^2(t)\Big)x_0(t)\Big]dt+P(t)H(t)dW(t).
\end{aligned}
\end{equation}
Comparing the coefficients with \eqref{o4}, we obtain
\begin{equation}\label{o21}\left\{\begin{aligned}
&C(t)=\alpha-B^2(t)R^{-1}(t)\pi_{12}(t)+P(t)H^2(t),\\
&D(t)=A(t)-B^2(t)R^{-1}(t)\big(\pi_{11}(t)+\pi_{13}(t)\big)-P(t)H^2(t),\\
&g(t)=-B^2(t)R^{-1}(t)\gamma_1(t)+m(t).
\end{aligned}\right.
\end{equation}
Thus, noting $f(\cdot),g(\cdot)$ and $\gamma(T)=0$ in \eqref{o17}, we get that $\gamma(\cdot)$ is deterministic. Then the uniqueness of BSDE implies $q(\cdot)\equiv0.$ Then we have
\begin{theorem}
Suppose \emph{(H1)-(H2)} hold true and the following Riccati equation system
\begin{equation}\label{o22}\left\{\begin{aligned}
&\dot{P}(t)-2A(t)P(t)+H^2(t)P^2(t)-(\sigma^2(t)+\tilde{\sigma}^2(t))=0,\ P(0)=0,\\
&\dot{\pi}(t)+\pi(t)\mathbb{A}(t)+\mathbb{A}'(t)\pi(t)-R^{-1}\pi(t)\mathbb{B}(t)\mathbb{B}'(t)\pi(t)+\mathbb{Q}(t)=0,\ \pi(T)=\mathbb{G},\\
&\dot{\gamma}(t)+\big(\mathbb{A}'(t)-R^{-1}(t)\pi(t)\mathbb{B}(t)\mathbb{B}'(t)\big)\gamma(t)+\pi(t)\mathbb{F}(t)=0,\ \gamma(T)=0
\end{aligned}\right.
\end{equation}admits an unique solution $(P(\cdot),\pi(\cdot),\gamma(\cdot))$
and $(\tilde{A}, \tilde{B}, C, D, f, g)$ are uniquely determined by \eqref{o19} and \eqref{o21}. Moreover, the SDE system
\begin{equation}\label{o23}\left\{\begin{aligned}
&dx_0(t)=\Big[\big(A(t)+\alpha-B^2(t)R^{-1}(t)\pi_{12}(t)\big)x_0(t)-B^2(t)R^{-1}(t)\big(\pi_{11}(t)+\pi_{13}(t)\big)\hat{x}(t)\\
&\qquad\quad-B^2(t)R^{-1}(t)\gamma_1(t)+m(t)\Big]+\tilde{\sigma}(t)dW(t),\\
&d\hat{x}(t)=\Big[\Big(\alpha-B^2(t)R^{-1}(t)\pi_{12}(t)+P(t)H^2(t)\Big)x_0(t)-B^2(t)R^{-1}(t)\gamma_1(t)+m(t)\\
&\qquad\quad+\Big(A(t)-B^2(t)R^{-1}(t)\big(\pi_{11}(t)+\pi_{13}(t)\big)-P(t)H^2(t)\Big)\hat{x}(t)\Big]dt+P(t)H(t)dW(t),\\
&x_0(0)=\hat{x}(0)=x
\end{aligned}\right.
\end{equation}admits an unique solution $(x_0(\cdot),\hat{x}(\cdot))$ and the decentralized optimal strategies are given by \eqref{o18}.\end{theorem}
The proof is similar to that of Section 3 and omitted.
\begin{remark}
\emph{\quad Similar to Section 3, the resultant $(\bar{u}_1,\bar{u}_2,\cdots,\bar{u}_N)$ satisfies the $\epsilon$-Nash equilibrium for Problem {(\textbf{PO})}. The equation system \eqref{o22}-\eqref{o23} combined will be called the consistency condition. This is similar to the analysis in \cite{H10} and we know the solvability of the consistency condition \eqref{o22}-\eqref{o23} is essentially reduced to the solvability of \eqref{o22}. Once a consistent solution to \eqref{o22} is obtained, a consistent solution to SDEs \eqref{o23} is guaranteed.}
\end{remark}The authors would like to thank the editors and reviewers for their valuable and efficient work.\\

\section{Appendix}
\textbf{Appendix A. Proof of Lemma \ref{xNhat}.} By \eqref{DHS} and \eqref{X4}, we have
\begin{equation}\nonumber\left\{\begin{aligned}
&d\Big(\hat{\bar{x}}^{(N)}(t)-\mathbb{E}x_0(t)\Big)=\Big(A(t)-B^2(t)R^{-1}(t)P(t)\Big)\Big(\hat{\bar{x}}^{(N)}(t)-\mathbb{E}x_0(t)\Big)dt\\
&\qquad\qquad\qquad\qquad +\frac{1}{N}\sigma(t)\sum_{i=1}^NdW_i(t),\\
&\hat{\bar{x}}^{(N)}(0)-\mathbb{E}x_0(0)=x^{(N)}(0)-x.
\end{aligned}\right.
\end{equation}
Thus
\begin{equation}\nonumber\begin{aligned}
\Big|\hat{\bar{x}}^{(N)}(t)-\mathbb{E}x_0(t)\Big|^2\leq &3\Big|x^{(N)}(0)-x\Big|^2+3\int_0^t\Big|A(s)-B^2(s)R^{-1}(s)P(s)\Big|^2\\
&\cdot\Big|\hat{\bar{x}}^{(N)}(s)-\mathbb{E}x_0(s)\Big|^2ds+3\Big|\int_0^t\frac{1}{N}\sigma(s)\sum_{i=1}^NdW_i(s)\Big|^2.
\end{aligned}
\end{equation}By the independence of $\{W_i(t)\}_{t\ge 0},1\le i\le N$, we have
\begin{equation}\nonumber\begin{aligned}
\mathbb{E}\Big|\int_0^t\frac{1}{N}\sigma(s)\sum_{i=1}^NdW_i(s)\Big|^2=O\Big(\frac{1}{N}\Big).
\end{aligned}
\end{equation}So \eqref{e1} follows by Gronwall's inequality. Combining with \eqref{e1}, the assertion \eqref{e2} can be proved in a similar way.   \hfill  $\Box$\\

\textbf{Appendix B. Proof of Lemma \ref{l8}.} Similar to the proof of Lemma \ref{yN}, we have
\begin{equation}\nonumber\begin{aligned}
&\sup_{0\leq t\leq T}\mathbb{E}\Big||y_i(t)-y^{(N)}(t)|^2-|\bar{x}_i(t)-x_0(t)|^2\Big|\\
\leq&\sup_{0\leq t\leq T}\mathbb{E}\Big|y_i(t)-\bar{x}_i(t)+x_0(t)-y^{(N)}(t)\Big|^2\\
+&2\Big(\sup_{0\leq t\leq T}\mathbb{E}|\bar{x}_i(t)-x_0(t)|^2\Big)^{\frac{1}{2}}\Big(\sup_{0\leq t\leq T}\mathbb{E}|y_i(t)-\bar{x}_i(t)+x_0(t)-y^{(N)}(t)|^2\Big)^{\frac{1}{2}}.
\end{aligned}
\end{equation}By Lemma \ref{yN}, it follows that
\begin{equation}\nonumber\begin{aligned}
&\sup_{0\leq t\leq T}\mathbb{E}\Big|y_i(t)-\bar{x}_i(t)+x_0(t)-y^{(N)}(t)\Big|^2\\\leq& \sup_{0\leq t\leq T}\mathbb{E}\Big|y_i(t)-\bar{x}_i(t)\Big|^2+\sup_{0\leq t\leq T}\mathbb{E}\Big|y^{(N)}(t)-x_0(t)\Big|^2\\
=&O\Big(\frac{1}{N}\Big).
\end{aligned}
\end{equation}Then, combining with the fact that $\sup\limits_{0\leq t\leq T}\mathbb{E}|\bar{x}_i(t)-x_0(t)|^2<+\infty$, we obtain
\begin{equation}\nonumber\begin{aligned}
\sup_{0\leq t\leq T}\mathbb{E}\Big||y_i(t)-y^{(N)}(t)|^2-|\bar{x}_i(t)-x_0(t)|^2\Big|=O\Big(\frac{1}{\sqrt{N}}\Big).
\end{aligned}
\end{equation}Thus,\begin{equation}\nonumber\begin{aligned}
&\Big|\mathcal{J}_i(\bar{u}_i,\bar{u}_{-i})-J_i(\bar{u}_i)\Big|\\
\leq&\mathbb{E}\int_0^T\Big|Q(t)(y_i(t)-y^{(N)}(t))^2-Q(t)(\bar{x}_i(t)-x_0(t))^2\Big|dt+\mathbb{E}\Big|Gy_i^2(T)-G\bar{x}_i^2(T)\Big|\\
=&O\Big(\frac{1}{\sqrt{N}}\Big).
\end{aligned}
\end{equation}The assertion \eqref{e6} follows.   \hfill  $\Box$\\

\textbf{Appendix C. Proof of Proposition \ref{z111}.} By \eqref{y2} and \eqref{y3}, it holds that
\begin{equation}\nonumber\begin{aligned}
|z_i(t)|^2\leq& 4|x_i(0)|^2+4K\int_0^t\Big[|z_i(s)|^2+|u_i(s)|^2+\frac{1}{N}\sum_{k=1}^N|z_k(s)|^2+|m(s)|^2\Big]ds\\
&+4\Big|\int_0^t\sigma(s)dW_i(s)\Big|^2+4\Big|\int_0^t\tilde{\sigma}(s)dW(s)\Big|^2
\end{aligned}
\end{equation}
and for $j\neq i$,
\begin{equation}\nonumber\begin{aligned}
|z_j(t)|^2\leq& 4|x_j(0)|^2+4K\int_0^t\Big[|z_j(s)|^2+|\bar{u}_j(s)|^2+\frac{1}{N}\sum_{k=1}^N|z_k(s)|^2+|m(s)|^2\Big]ds\\
&+4\Big|\int_0^t\sigma(s)dW_j(s)\Big|^2+4\Big|\int_0^t\tilde{\sigma}(s)dW(s)\Big|^2
\end{aligned}
\end{equation}where $K:=\max\limits_{0\leq t\leq T}\Big(A^2(t)+B^2(t)\Big)+\alpha^2$. Thus,
\begin{equation}\nonumber\begin{aligned}
\mathbb{E}\Big[\sum_{k=1}^N|z_k(t)|^2\Big]\leq & 4\mathbb{E}\Big[\sum_{k=1}^N|x_k(0)|^2\Big]+4K\mathbb{E}\int_0^t\Big[2\sum_{k=1}^N|z_k(s)|^2+|u_i(s)|^2+\sum_{k=1,k\neq i}^N|\bar{u}_k(s)|^2\\
&+N|m(s)|^2\Big]ds+4\sum_{k=1}^N\mathbb{E}\Big|\int_0^t\sigma(s)dW_k(s)\Big|^2+4N\mathbb{E}\Big|\int_0^t\tilde{\sigma}(s)dW(s)\Big|^2.
\end{aligned}
\end{equation}By \eqref{ubound}, we can see that $\mathbb{E}|u_i(t)|^2$ is bounded. Besides, the optimal controls $\bar{u}_k(t),k\neq i$ are $L^2$-bounded. Then by Gronwall's inequality, it follows that
\begin{equation}\nonumber
\sup_{0\leq t\leq T}\mathbb{E}\Big[\sum_{k=1}^N|z_k(t)|^2\Big]=O(N),
\end{equation}
and $\sup\limits_{0\leq t\leq T}\mathbb{E}|z_i(t)|^2$ is bounded.  \hfill $\Box$\\

\textbf{Appendix D. Proof of Proposition \ref{allz}.} From \eqref{y7}, it follows that
\begin{equation}\label{zN}\left\{\begin{aligned}
d\check{z}^{(N-1)}(t)=&\Big[(A(t)+\frac{N-1}{N}\alpha)\check{z}^{(N-1)}(t)-B^2(t)R^{-1}(t)\\
&\cdot\Big(P(t)\hat{\bar{x}}^{(N-1)}(t)+\hat{P}(t)\mathbb{E}x_0(t)+\Phi(t)\Big)+m(t)\Big]dt\\
&+\frac{1}{N-1}\sum_{j=1,j\neq i}^N\sigma(t)dW_j(t)+\tilde{\sigma}(t)dW(t),\\
\check{z}^{(N-1)}(0)=&x^{(N-1)}(0).
\end{aligned}\right.
\end{equation}
Then we have
\begin{equation}\nonumber\left\{\begin{aligned}
&d\Big(\check{z}^{(N-1)}(t)-z^{(N-1)}(t)\Big)=\Big[(A(t)+\frac{N-1}{N}\alpha)\Big(\check{z}^{(N-1)}(t)-z^{(N-1)}(t)\Big)-\frac{\alpha}{N}z_i(t)\Big]dt,\\
&\check{z}^{(N-1)}(0)-z^{(N-1)}(0)=0.
\end{aligned}\right.
\end{equation}
By the $L^2$-boundness of $z_i(t)$ and Gronwall's inequality, the assertions \eqref{zN1x0} and \eqref{zizi} hold.
And by \eqref{e1}, we can get \eqref{zzN}. Besides, it follows that
\begin{equation}\nonumber
\sup_{0\leq t\leq T}\mathbb{E}\Big|\check{z}^{(N-1)}(t)-x_0(t)\Big|^2=O\Big(\frac{1}{N}\Big)
\end{equation}
 from \eqref{zN}, \eqref{DHS} and \eqref{zzN}. Then \eqref{y4}, \eqref{y6} and \eqref{yN1Ex0} imply that
\begin{equation}\nonumber
\sup_{0\leq t\leq T}\mathbb{E}\Big|\check{z}_i(t)-\bar{x}^0_i(t)\Big|^2=O\Big(\frac{1}{N}\Big).
\end{equation}
Finally, by Proposition \ref{z111}, we easily get \eqref{zzN1}.   \hfill  $\Box$\\

\textbf{Appendix E. Proof of Lemma \ref{p3}.}  \eqref{Le1} and \eqref{Le2} follow from Proposition \ref{allz} directly. By Proposition \ref{z111}, we get that both $\sup_{0\leq t\leq T}\mathbb{E}|\bar{x}^0_i(t)|^2$ and $\sup_{0\leq t\leq T}\mathbb{E}|\bar{x}^0_i(t)-x_0(t)|^2$ are bounded. Similar to the proof of Lemma \ref{yN}, we have
\begin{equation}\nonumber\begin{aligned}
&\sup_{0\leq t\leq T}\mathbb{E}\Big||z_i(t)|^2-|\bar{x}^0_i(t)|^2\Big|\\
\leq& \sup_{0\leq t\leq T}\mathbb{E}|z_i(t)-\bar{x}^0_i(t)|^2+2\sup_{0\leq t\leq T}\mathbb{E}|\bar{x}^0_i(t)(z_i(t)-\bar{x}^0_i(t))|\\
\leq&\sup_{0\leq t\leq T}\mathbb{E}|z_i(t)-\bar{x}^0_i(t)|^2+2\Big(\sup_{0\leq t\leq T}\mathbb{E}|\bar{x}^0_i(t)|^2\Big)^{\frac{1}{2}}\Big(\sup_{0\leq t\leq T}\mathbb{E}|z_i(t)-\bar{x}^0_i(t)|^2\Big)^{\frac{1}{2}}\\
=&O\Big(\frac{1}{\sqrt{N}}\Big),
\end{aligned}
\end{equation}
thus \eqref{Le3} holds. Besides,
\begin{equation}\nonumber\begin{aligned}
&\sup_{0\leq t\leq T}\mathbb{E}\Big||z_i(t)-z^{(N)}(t)|^2-|\bar{x}^0_i(t)-x_0(t)|^2\Big|\\
\leq&\sup_{0\leq t\leq T}\mathbb{E}\Big|z_i(t)-\bar{x}^0_i(t)+x_0(t)-z^{(N)}(t)\Big|^2\\
&+2\Big(\sup_{0\leq t\leq T}\mathbb{E}|\bar{x}^0_i(t)-x_0(t)|^2\Big)^{\frac{1}{2}}\Big(\sup_{0\leq t\leq T}\mathbb{E}|z_i(t)-\bar{x}^0_i(t)+x_0(t)-z^{(N)}(t)|^2\Big)^{\frac{1}{2}}\\
=&O\Big(\frac{1}{\sqrt{N}}\Big),
\end{aligned}
\end{equation}
then \begin{equation}\nonumber\begin{aligned}
&\Big|\mathcal{J}_i(u_i,\bar{u}_{-i})-J_i(u_i)\Big|\\
\leq&\mathbb{E}\int_0^T\Big|Q(t)(z_i(t)-z^{(N)}(t))^2-Q(t)(\bar{x}^0_i(t)-x_0(t))^2\Big|dt+\mathbb{E}\Big|Gz_i^2(T)-G(\bar{x}^0_i(T))^2\Big|\\
=&O\Big(\frac{1}{\sqrt{N}}\Big),
\end{aligned}
\end{equation}
which implies \eqref{Le4}.  \hfill  $\Box$


\begin{thebibliography}{99}

\bibitem{AD} D. Andersson, B. Djehiche. A maximum principle for stochastic control
of SDE's of mean-field type. \emph{Applied Mathematics and
Optimization}, \textbf{63}, 341-356, 2010.

\bibitem{Anto} F. Antonelli. Backward-forward stochastic differential equations. \emph{Ann. Appl. Probab.}, \textbf{3}, 777-793, 1993.

\bibitem{BO07}
F. Baghery and B. \O ksendal.  A maximum principle for stochastic control with partial information.  {\it Stochastic Anal. Appl.}, \textbf{25}, 705-717, 2007.

\bibitem{BEK}
J. Baras, R. Elliott and M. Kohlmann. The partially observed stochastic minimum principle. {\it SIAM J. Control Optim.}, \textbf{27}, 1279-1292, 1989.

\bibitem{B12}
M. Bardi. Explicit solutions of some linear-quadratic mean field games.
{\it Networks and Heterogeneous Media}, \textbf{7}, 243-261, 2012.

\bibitem{Ben83}
A. Bensoussan. Maximum principle and dynamic programming approaches of the optimal control of partially observed diffusions. {\it Stochastics}, \textbf{9}, 169-222, 1983.

\bibitem{Ben92}
A. Bensoussan. Stochastic control of partially observable systems. Cambridge University Press, Cambridge, 1992.

\bibitem{B11}
A. Bensoussan, K. Sung, S. Yam and S. Yung. Linear-quadratic mean-field games. Preprint, 2014.

\bibitem{BDL} R. Buckdahn, B. Djehiche and J. Li. A general stochastic
maximum principle for SDEs of mean-field type. \emph{Applied Mathematics
and Optimization}, \textbf{64}, 197-216, 2011.

\bibitem{CD} R. Carmona and F. Delarue. Probabilistic analysis of mean-field games. \emph{SIAM J. Control Optim.}, \textbf{51}, 2705-2734, 2013.

\bibitem{cfs} R. Carmona, J. P. Fouque and L. H. Sun.
Mean Field Games and Systemic Risk, to appear in \emph{Communications in Mathematical Sciences}, 2014.

\bibitem{dc} F. Delarue and R. Carmona. Probabilistic Approach for Mean Field Games
with a Common Noise, Mean Field Games and Related Topics, Padova, 2013.


\bibitem{dh} B. Djehiche and M. Huang. A characterization of sub-game perfect Nash equilibria for SDEs of mean field type. Preprint, 2014.


\bibitem{et} G. Espinosa and N. Touzi. Optimal investment under relative performance concerns. \emph{Mathematical Finance}, 2013.


\bibitem{gll} O. Gueant, J. M. Lasry and P. L. Lions. Mean field games and applications. Paris-Princeton Lectures on Mathematical Finance 2010. Lecture Notes in Mathematics Volume 2003, 205-266, 2011.

\bibitem{H10} M. Huang. Large-population LQG games involving a major player: the Nash certainty equivalence principle. \emph{SIAM J. Control Optim.}, \textbf{48}, 3318-3353, 2010.

\bibitem{hcm06} M. Huang, P. E. Caines and R. P. Malham\'{e}. Distributed multi-agent decision-making with partial observations: asymptotic Nash equilibria. \emph{Proceedings of the 17th International Symposium on Mathematical Theory of Networks and Systems}, Kyoto, Japan, 2006.

\bibitem{hcm07} M. Huang, P. E. Caines and R. P. Malham\'{e}. Large-population cost-coupled LQG problems with non-uniform agents: individual-mass behavior and decentralized $\varepsilon$-Nash
equilibria. \emph{IEEE Transactions on Automatic Control}, \textbf{52}, 1560-1571, 2007.


\bibitem{HCM12} M. Huang, P. E. Caines and R. P. Malham\'{e}. Social optima in mean field LQG control: centralized and decentralized strategies.  \emph{IEEE Transactions on Automatic Control}, \textbf{57}, 1736-1751, 2012.

\bibitem{hmc06} M. Huang, R. P. Malham\'{e} and P. E. Caines. Large population stochastic dynamic games: closed-loop McKean-Vlasov systems and the Nash certainty equivalence principle. \emph{Communication in Information and Systems}, \textbf{6}, 221-251, 2006.

\bibitem{KJ} N. El Karoui and M. Jeanblanc-Picqu\'e. Optimization of consumption with labor income. {\it Finance Stochast.}, \textbf{2}, 409-440, 1998.

\bibitem{KPQ}
N. El Karoui, S. Peng, and M. Quenez. Backward stochastic differential equations in finance. \emph{Math. Finance}, \textbf{7}, pp. 1-71, 1997.

\bibitem{KC13}A. C. Kizilkale and P. E. Caines. Mean field stochastic adaptive control. {\it IEEE Transactions on Automatic Control}, \textbf{58}, 905-920, 2013.

\bibitem{ll} J. M. Lasry and P. L. Lions. Mean field games. \emph{Japan J. Math.}, \textbf{2}, 229-260, 2007.

\bibitem{LZ08}
T. Li and J. Zhang. Asymptotically optimal decentralized control
for large population stochastic multiagent systems. {\it IEEE Transactions on Automatic Control}, \textbf{53}, 1643-1660, 2008.

\bibitem{my}  J. Ma and J. Yong. Forward-backward stochastic differential equations and their applications. Springer-Verlag, Berlin Heidelberg, 1999.

\bibitem{MOZ} T. Meyer-Brandis, B. ${\O}$ksendal and X. Y. Zhou. A mean-field stochastic maximum principle via Malliavin
calculus. \emph{Stochastics}, \textbf{84}, 643-666, 2012.


\bibitem{NH12}
S. L. Nguyen and M. Huang. Linear-quadratic-Gaussian mixed games with continuum-parametrized
minor players.  {\it SIAM J. Control Optim.}, \textbf{50}, 2907-2937, 2012.

\bibitem{NCMH12}
M. Nourian,  P. E. Caines, R. P. Malham\'e and M. Huang.
Mean field control in leader-follower stochastic
multi-agent systems: likelihood ratio based adaptation. {\it IEEE Transactions on Automatic Control}, \textbf{57}, 2801-2816, 2012.

\bibitem{TST}
I. Tahar, H. Soner and N. Touzi. The dynamic programming equation for the problem of optimal investment under capital gains taxes. {\it SIAM J. Control Optim.}, \textbf{46}, 1779-1801, 2008.

\bibitem{Tang98}
S. Tang.  The maximum principle for partially observed optimal control of stochastic differential equations. {\it SIAM J. Control Optim.}, \textbf{36}, 1596-1617, 1998.

\bibitem{TZB11}
H. Tembine, Q. Zhu and T. Basar. Risk-sensitive mean-field
stochastic differential games. {\it Proc. 18th IFAC World Congress},
Milan, Italy, Aug. 2011.

\bibitem{WW2009}
G. Wang and Z. Wu. The maximum principles for stochastic recursive optimal control problems under partial information. {\it IEEE Transactions on Automatic Control}, \textbf{54}, 1230-1242, 2009.

\bibitem{xz} J. Xiong and X. Y. Zhou. Mean-variance portfolio selection under partial information. {\it SIAM J. Control Optim.}, \textbf{46}, 156-175, 2007.

\bibitem{Y02} J. Yong. A leader-following stochastic linear quadratic differential game. {\it SIAM J. Control Optim.}, \textbf{41}, 1015-1041, 2002.

\bibitem{yong 2011}  J. Yong. A linear-quadratic optimal control problem for mean-field stochastic differential equations. \emph{SIAM J. Control Optim.}, \textbf{51}, 2809-2838, 2013.

\bibitem{yz}  J. Yong and X. Y. Zhou. Stochastic controls: Hamiltonian systems and HJB equations. Springer-Verlag, New York, 1999.

\bibitem{Zhou93}
X. Y. Zhou.  On the necessary conditions of optimal controls for stochastic partial differential equations. {\it SIAM J. Control Optim.}, \textbf{31}, 1462-1478, 1993.


\end{thebibliography}
\end{document}